\documentclass[envcountsame, runningheads]{llncs}
\usepackage{amsmath, amssymb}
\usepackage{tikz}
\usepackage{subcaption}
\spnewtheorem{observation}[theorem]{Observation}{\bfseries}{\itshape}
\usetikzlibrary{positioning, shapes, arrows.meta, calc, intersections}
\usetikzlibrary{decorations.pathmorphing, positioning, calc}
\usepackage[justification=centering]{caption}
\usepackage{comment}
\usepackage{hyperref}
\usepackage{graphicx}      
\usepackage{caption}       
\usepackage{subcaption}    
\usepackage{float}     
\usepackage{xcolor}
\usepackage{wrapfig,lipsum}
\usepackage{lineno}

\title{Conflict-Free Cuts in Planar and 3-Degenerate Graphs with 1-Regular Conflicts\thanks{A preliminary version of this paper appeared in IWOCA 2026 \cite{iwoca-cfcuts}.}}
\titlerunning{Conflict-Free Cuts in Planar and 3-Degenerate Graphs}
\author{Subrahmanyam Kalyanasundaram\thanks{The author acknowledges support from ANRF through project number ANRF/ARGM/2025/002583/MTR.}
\and 
Subodh Kumar
\institute{Department of Computer Science and Engineering, IIT Hyderabad \\
\email{\{subruk, cs23resch11009\}@iith.ac.in}
}
}
\date{}

\begin{document}

\maketitle

\begin{abstract}
A conflict-free cut $F$ on a simple connected graph $G = (V, E)$ is defined as a set of edges $F \subseteq E$ such that $G-F$ is disconnected, and no two edges in $F$ are conflicting. The notion of conflicting edges is represented using an associated conflict graph $\widehat{G} = (\widehat{V}, \widehat{E})$ where $\widehat{V} = E$. Deciding if a given planar graph $G$, with an associated conflict graph $\widehat{G}$, has a conflict-free cut is known to be NP-complete, when $G$ has maximum degree four and $\widehat{G}$ is a line graph of $G$ [Bonsma, JGT 2009].  

In this paper, we prove the following for the case when $\widehat{G}$ is 1-regular.
\begin{itemize}
    \item We completely resolve the complexity of the decision problem when $G$ is planar. Towards this end, we show that (a) there always exists a conflict-free cut when the graph is planar and 4-regular unless it is the octahedron graph and (b) the decision problem is NP-complete, even in the case when $G$ is planar with maximum degree 5.
    \item We also show that the decision problem is NP-complete
    when $G$ is a 3-degenerate graph with maximum degree 5. This completely resolves the complexity status of the problem when $G$ is 3-degenerate.

    \item We construct families of graphs with 1-regular conflict graphs that do not have a conflict-free cut.
\end{itemize}
Our results answer the questions posed in [Rauch, Rautenbach and Souza, IPL 2025].

\end{abstract}


\section{Introduction}

We study the conflict-free cut problem for planar and 3-degenerate graphs, when the conflict graph $\widehat{G}$ is 1-regular. The decision version of the problem is formally defined below:
\begin{definition}[CF-cut problem]
    Given a simple connected graph $G=(V,E)$ and a conflict graph $\widehat{G}=(\widehat{V}, \widehat{E})$ with $\widehat{V} = E$, the \emph{Conflict-Free Cut (CF-cut)} problem is to decide if there is a set $F \subseteq E$ such that $G-F$ is disconnected and $F$ is independent in $\widehat{G}$.
\end{definition}
The notion of conflicts on edges is represented using the  conflict graph $\widehat{G} = (\widehat{V}, \widehat{E})$ where $\widehat{V} = E$. Edges  $e_i$ and $e_j$ conflict in $G$ if and only if there is an edge $\{e_i, e_j\} \in \widehat{E}$.

Several optimization problems have been studied under conflict constraints. The conflict constraints restrict the solution space by assigning conflicts which forbid items to be chosen simultaneously. Some examples include spanning tree problem \cite{barros2025conflict,darmann2011}, 
 bin packing \cite{conflict-binpacking},
knapsack \cite{knapsack-conflict1,knapsack-conflict2},
matching \cite{glock2024conflict,agrawal2020parameterized,darmann2011}, 
and feedback vertex set \cite{conflict-FVS}. Along the same lines, the CF-cut problem was introduced by  Rauch, Rautenbach, and Souza in \cite{rauch2025conflict}.

We obtain the matching cut problem as a special case of the CF-cut problem when the conflict graph $\widehat{G}$ is the line graph of $G$. The \emph{line graph} of $G$ is given by $L(G)$, where $V(L(G)) = E(G)$ and $\{e_i, e_j\} \in E(L(G))$ if and only if $e_i$ and $e_j$ share an endpoint in $G$.
First discussed in \cite{graham1971bounds}, the matching cut problem is a well-studied problem in the literature.
Matching cut is known to be NP-complete in general \cite{Chvatal1984}, and even in the case when $G$ is planar \cite{bonsma2009complexity}.
The reader is referred to the following papers  \cite{komusiewicz2020,chen2021,golovach2021,feghali2025journal,Lucke2024journal} for more results on matching cuts.


In this paper, we study the conflict-free cut problem in the case where the conflict graph  $\widehat{G}$ is 1-regular. Apart from being a natural restriction, this case has been studied for other problems with conflict constraints. 
The conflict-free variant of maximum matching is NP-complete under the 1-regular constraint, but the spanning tree problem is polynomial time solvable \cite{darmann2011}.
For the matching cut problem on a graph, an edge $e$ conflicts with as many as $(d_1 + d_2 - 2)$ edges where $d_1$ and $d_2$ are  degrees of the end points of $e$. In other words, the degree of $e$ in the conflict graph $\widehat{G}$ is $(d_1 + d_2 - 2)$. An immediate question is: what is the complexity of conflict-free cut when $\widehat{G}$ is much simpler, say, 1-regular?

In the paper that introduced the CF-cut problem, Rauch, Rautenbach and Souza \cite{rauch2025conflict} showed that the CF-cut problem is NP-complete even when the maximum degree of $G$ is 5 with a 1-regular $\widehat{G}$. They also studied the complexity of the problem in the setting of parameterized complexity, showing that it is fixed parameter tractable with respect to the vertex cover number of $G$, and hard with respect to the size of the feedback vertex set of $G$, and the clique cover number of $\widehat{G}$. 
One natural direction here is the complexity of the CF-cut problem when $G$ is planar, and has maximum degree 5 with 1-regular $\widehat{G}$. The authors of \cite{rauch2025conflict} also asked about the complexity of the problem when $G$ is 3-degenerate with 1-regular $\widehat{G}$. In this paper, we show that both the problems are NP-complete.

When the conflict-graph
$\widehat{G}$ is 1-regular, every edge in $G$ conflicts with exactly one other edge. 
Proposition 4.1 in \cite{rauch2025conflict} notes the following 
when $\widehat{G}$ is 1-regular. 
If the edges incident with a vertex $v \in V(G)$  do not have conflicts among each other, then we have a CF-cut that separates $v$ from the rest of the graph.
Hence, if  $|E(G)|< 2|V(G)|$, then there is a CF-cut in $G$.



Since we show NP-completeness when $G$ is planar and has maximum degree 5, it is natural to study the complexity when 
$G$ is planar and has maximum degree 4. For a graph $G$ with maximum degree 4 that is not 4-regular, it follows that  $|E(G)| < 2 |V(G)|$ and hence there always exists a CF-cut,  when the conflict graph $\widehat{G}$ is 1-regular. 
In this paper, we study the case when $G$ is planar and 4-regular, and show that it always contains a CF-cut unless it is the octahedron graph. 

We present the following results in this paper.

\begin{enumerate}
\item Given a 4-regular planar graph $G$ with  a 1-regular  $\widehat{G}$, there always exists a CF-cut unless it is the octahedron graph. This is shown in Theorem \ref{thm:4-regular-planar}.
\item Given a planar graph $G$ of maximum degree 5 with a 1-regular $\widehat{G}$, it is NP-complete to decide if $G$ has a CF-cut. This is proved in Theorem \ref{thm:planarNPC}.

This result, along with the result in point 1, completely resolves the complexity of the problem when $G$ is planar with 1-regular  $\widehat{G}$.

\item Given a 3-degenerate graph $G = (V, E)$ of maximum degree 5 with a 1-regular $\widehat{G}$, it is NP-complete to decide if $G$ has a CF-cut. This is proved in Theorem \ref{thm:3-degenerate_graph} and also answers a question by \cite{rauch2025conflict} (stated as Problem 4.3).

It should be noted that when $G$ is 3-degenerate with maximum degree 4, it must necessarily contain a CF-cut since $|E(G)| < 2 |V(G)|$. Hence, our result
completely settles the 3-degenerate case.

\item We present a family of uncuttable graphs $G$ with 1-regular $\widehat{G}$ such that $|E(G)| = 2|V(G)|$. This is proved in Lemma \ref{lem:uncuttableH''}. 

We also present another uncuttable family of graphs that are triangulations, one of which is the octahedron graph (depicted in Figure~\ref{fig:octahedron}). This is proved in Theorem \ref{thm:easy_planar}. These families  address  questions posed by \cite{rauch2025conflict} (stated as Problems 4.1 and 4.2).


\end{enumerate}

\section{Preliminaries}
Unless otherwise mentioned, we use $G$ to denote a simple, connected graph.
For a graph $G = (V, E)$ and two disjoint sets of vertices $A, B \subseteq V$, we use the notation $E(A, B)$ to describe the set of edges of $G$ with one vertex in $A$ and another in $B$. We define $M \subseteq E$ as a minimal edge cut of $G$ if $V = A\cup B$ such that $A$ and $B$ induce connected subgraphs, and $M = E(A, B)$. 
We sometimes use the shorthand notation $xy$ to denote an edge $\{x, y\}$.
 For a graph $G = (V, E)$, and a vertex set $S \subseteq V$, we use $G[S]$ to represent the induced graph on $S$. In this paper, we use the term cut, to refer only to an edge cut. We use standard graph-theoretic notation as presented in the book by Diestel \cite{diestel2005graph}.
We refer to a graph without any CF-cut as an \emph{uncuttable graph}, or simply \emph{uncuttable}. Sometimes, for a graph $G$ with 1-regular $\widehat{G}$, we use the term: a graph with 1-regular conflicts. 

We note some definitions that will be needed.
\begin{itemize}
    \item A graph is \emph{$k$-regular} if every vertex in the graph has degree $k$.
    \item A graph is said to be \emph{planar} if it can be drawn in a plane with no two edges crossing each other. For a planar graph, such a drawing is called a \emph{planar embedding}.
    \item An undirected graph $G$ is \emph{$k$-degenerate} if  there is an ordering of the vertices of $G$ such that each vertex $x$ has at most $k$ neighbors  that precede $x$ in the ordering.
    
    From the above definition, it follows that $G$ is $k$-degenerate if and only if we can successively delete a vertex of degree at most $k$ eventually resulting in an empty graph.  
    \item Given a graph $G = (V, E)$, the \emph{square graph of $G$} is defined as $G^2 = (V, \widetilde{E})$ as a graph on the same set of vertices and a superset $\widetilde{E} \supseteq E$ of edges. The set of edges $\widetilde{E}$ is all the pairs of vertices $x, y \in V$ such that either $xy \in E$ or there exists a vertex $w \in V$ such that both $xw, yw \in E$.  
\end{itemize}


\section{4-regular Planar Graphs} \label{sec:4-reg-planar}
In this section, we study the CF-cut problem when $G$ is planar and 4-regular. As noted before, if $|E(G)| < 2|V(G)|$, there always exists a CF-cut when $\widehat{G}$ is 1-regular. This includes the case when $G$ has maximum degree 4, but is not 4-regular. The complexity of the problem for the 4-regular case is not known. In this section, we show that when $G$ is 4-regular and planar, there always exists a CF-cut except in the case when $G$ is the octahedron graph (Figure~\ref{fig:octahedron}).

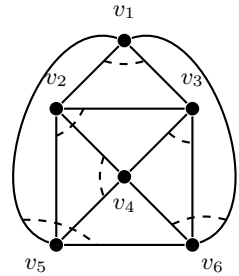
\begin{wrapfigure}{r}{5.5cm}
\centering
\begin{tikzpicture}[scale=0.9, thick, every node/.style={circle, fill=black, inner sep=1.8pt, font=\small}]

\node (v1) at (0,4) [label=above: $v_1$] {};
\node (v2) at (-1,3) [label=above: $v_2$] {};
\node (v3) at (1,3) [label=above:$v_3$] {};
\node (v4) at (0,2) [label=below:$v_4$] {};
\node (v5) at (-1,1) [label=below left:$v_5$] {};
\node (v6) at (1,1) [label=below right:$v_6$] {};

\draw (v1)--(v2);
\draw (v1)--(v3);
\draw (v2)--(v3);
\draw (v2)--(v4);
\draw (v3)--(v4);
\draw (v4)--(v5);
\draw (v4)--(v6);
\draw (v1) to[bend right=90] node[pos=0.85, coordinate] (m15) {} (v5);
\draw (v1) to[bend left=90]  node[pos=0.85, coordinate] (m16) {} (v6);
\draw (v2)--(v5);
\draw (v3)--(v6);
\draw (v5)--(v6); 

\coordinate (m12) at ($(v1)!0.3!(v2)$);
\coordinate (m13) at ($(v1)!0.3!(v3)$);
\coordinate (m25) at ($(v2)!0.2!(v5)$);
\coordinate (m23) at ($(v2)!0.2!(v3)$);
\coordinate (m24) at ($(v2)!0.7!(v4)$);
\coordinate (m45) at ($(v4)!0.3!(v5)$);
\coordinate (m34) at ($(v3)!0.35!(v4)$);
\coordinate (m36) at ($(v3)!0.25!(v6)$);
\coordinate (m56) at ($(v5)!0.3!(v6)$);
\coordinate (m64) at ($(v6)!0.3!(v4)$);

\draw[dashed, bend right=15] (m12) to (m13);   
\draw[dashed, bend right=20] (m25) to (m23);  
\draw[dashed, bend right=20] (m24) to (m45);   
\draw[dashed, bend right=20] (m34) to (m36);  
\draw[dashed, bend left=20] (m15) to (m56);   
\draw[dashed, bend left=20] (m64) to (m16);  

\end{tikzpicture}
\caption{Uncuttable octahedron with 1-regular conflicts.}
\label{fig:octahedron}
\end{wrapfigure}

\begin{theorem}\label{thm:4-regular-planar} Given a 4-regular planar graph $G$ with 1-regular conflict graph $\widehat{G}$, there always exists a CF-cut unless $G$ is the octahedron graph.
\end{theorem}

We give an overview of the proof. We first show that the planar embedding must contain faces that do not have conflicts between adjacent edges. If such a face is a triangle, we reason that there is always a CF-cut unless the graph is an octahedron. Otherwise, we can separate the face from the rest of the graph unless there are two conflicting edges that connect the face to the rest of the graph. We show that the graph contains a CF-cut in this case as well. The proof is through detailed case analysis. 

Throughout this section, we assume that the planar graph $G =(V,E)$ is given with an associated planar embedding. We may skip the explicit mention of the embedding. When a CF-cut contains only one vertex on one side  of the cut, and all the remaining vertices  on the other side, we refer to such a cut as a \emph{trivial CF-cut}. 



\begin{observation}\label{obs:one_pair_conflict}
In a 4-regular planar graph $G$ with 1-regular $\widehat G$ without a trivial CF-cut, each vertex is incident with exactly two mutually conflicting edges.
\end{observation}
\begin{proof}
    If the incident edges of any vertex do not have any conflicts with each other, we have a trivial CF-cut. Since we have $|E(G)| = 2|V(G)|$, each vertex must have exactly one conflict among its incident edges.
    \qed
\end{proof}

\begin{observation}
\label{obs:necessary-4regular}
For a 4-regular planar graph $G$ with 1-regular $\widehat{G}$ to be uncuttable, a necessary condition is that all the conflicts must be between adjacent edges (edges incident with the same vertex).
\end{observation} 

\begin{proof} 
There are exactly $|E(G)| /2 = |V(G)|$ conflicts. If any conflict is between two non-adjacent edges, then at least one vertex has no conflict among the incident edges, resulting in a trivial cut.
\qed
\end{proof}

For an uncuttable 4-regular planar graph $G = (V, E)$ with an associated planar embedding and a 1-regular $\widehat{G}$, we define two types of conflicts: (1) Type~1 ($T_1$) conflict when the pair of conflicting edges belongs to the same face,
and (2) Type~2 ($T_2$) conflict when the pair of conflicting edges does not belong to the same face.

\begin{observation}\label{obs:region-types-in-4reg-planar} For any 4-regular planar graph $G$ with an associated planar embedding and a 1-regular $\widehat{G}$, there can be only two types of faces: one that has conflicting boundary edges and the other that does not.
\end{observation}

The following definition is useful for the remainder of this proof.

\begin{definition}[CF face] 
\label{def:cfface}
Given a planar graph $G$ with planar embedding and 1-regular $\widehat{G}$, a face is said to be a \emph{conflict-free face} or \emph{CF face}, if no two boundary edges of the face are mutually conflicting. We explicitly refer to a CF face on three vertices as a \emph{CF triangle}.
\end{definition}

\begin{observation}\label{obs:conflicts_around_CF_face} For $\ell \ge 3$, let $v_0 v_1 v_2 \ldots v_{\ell-1} v_0$ be a CF face in an uncuttable 4-regular planar graph $G$ with an associated planar embedding and a 1-regular $\widehat{G}$. Then one of the following holds:
    \begin{itemize}
        \item For all $0 \le i \le \ell-1$, the boundary edge $v_i v_{(i+1) \bmod \ell}$ conflicts with an edge incident with $v_i$ (anticlockwise end point); or
        \item For all $0 \le i \le \ell-1$, the boundary edge $v_i v_{(i+1) \bmod \ell}$ conflicts with an edge incident with $v_{i+1}$ (clockwise end point).
    \end{itemize}
\end{observation}

\begin{proof}
    Since $G$ is planar and 4-regular, as per Observation \ref{obs:one_pair_conflict}, it follows that every vertex has exactly one pair of incident edges conflicting. Since the face is a CF face, each boundary edge conflicts with one of its adjacent edges that are not part of the face boundary.
\qed
\end{proof}

\begin{lemma} 
\label{lem:CF-regions-in-4regular-planar}
There always exist at least two CF faces in a 4-regular planar graph $G$ with 1-regular $\widehat{G}$.
\end{lemma}

\begin{proof} We know that there are $2\vert V(G) \rvert$ edges in a 4-regular graph $G$, and by Euler's theorem for a planar graph, there are $\lvert V(G) \rvert + 2$ faces. With 1-regular $\widehat{G}$, there are $\lvert V(G) \rvert$ conflicts in the graph $G$. Even when all conflicts are of $T_1$ type, i.e., two conflicting edges belong to a face, we have at least two faces without their boundary edges conflicting.\qed
\end{proof}

Based on the above lemma, we know that there is a CF face in  a 4-regular planar graph $G$. The following four lemmas are also necessary ingredients of the proof of Theorem \ref{thm:4-regular-planar}. 
Since the proofs of the lemmas are lengthy, 
we first prove Theorem \ref{thm:4-regular-planar}
using the lemmas, and then provide the full proofs of the lemmas.

\begin{lemma}
\label{lem:CF-trianle-in-4regular-planar}
Suppose $G$ is a 4-regular planar graph with 1-regular conflicts. If $G$ is uncuttable and has a CF triangle, then $G$ is the octahedron graph. 
\end{lemma}

\begin{proof}[Sketch of proof]
Each of the boundary edges of the CF triangle must conflict with another edge.  Using Observation \ref{obs:conflicts_around_CF_face}, we arrive at four possibilities on how many of these conflicts can be $T_1$ and how many can be $T_2$. A detailed case analysis implies that the only case where $G$ is uncuttable is when it is the octahedron graph.
\qed
\end{proof}

\begin{lemma} 
\label{lem:no_edges_on_CF_face} For $\ell \ge 4$, let $F= v_0 v_1 v_2 \ldots v_{\ell-1} v_0$ be a CF face in an uncuttable 4-regular planar graph $G$ with an associated planar embedding and a 1-regular $\widehat{G}$. Then there exist no edges whose one end point is at $v_i$ and the other end point is at $v_j$ for $|j - i| \bmod \ell \ge 2$.
\end{lemma}

\begin{proof}[Sketch of proof]
The edge $v_iv_j$ forms a cycle together with the path  $v_iv_{i+1}\ldots v_{j-1}v_j$. Suppose the other neighbors of $v_i$ and $v_j$ are $y$ and and $z$ respectively. The proof is divided into four cases depending on the     position of the edges $v_i y$ and $v_jz$. A CF-cut is demonstrated in each of these cases.
\qed
\end{proof}

\begin{lemma}
\label{lem:protect_from_far}
Suppose $G$ is a 4-regular planar graph with 1-regular conflicts and $v_0v_1v_2\ldots v_{\ell-1}v_0$, $\ell \ge 4$, be one of its CF faces. Let $v_i$ and $v_j$ be two vertices of this CF face that are not adjacent on this face. If $G$ has two conflicting edges, both incident with a common vertex, say $x$, and meeting the CF face at vertices $v_i$ and $v_j$ respectively, then there always exists a CF-cut. 
\end{lemma}

\begin{proof}[Sketch of proof]
Let $u$ and $w$ be the other neighbors of $x$. The sequence $xv_iv_{i+1}\ldots v_{j-1}v_j x$ forms a cycle. The proof is divided into three cases based on how many of the edges $xu$ and $xw$ are inside the region bounded by the cycle. In each case, we demonstrate a CF-cut.
\qed
\end{proof}

\begin{lemma}
\label{lem:CFFPT_case}
Suppose $G$ is a 4-regular planar graph with 1-regular conflicts and $F = v_0v_1v_2\ldots v_{\ell-1}v_0$, $\ell \ge 4$, be one of its CF faces. 
Suppose that for some $0 \leq i \leq \ell -1$,  the  vertices $v_i, v_{i+1}$, together with a vertex $x$ that is not part of $F$, form a triangle.
If the edges $xv_i$ and $xv_{i+1}$ conflict, then there  exists a CF-cut in $G$.
\end{lemma}

We are now ready to prove Theorem \ref{thm:4-regular-planar} using the above lemmas.

\begin{proof}[Proof of Theorem \ref{thm:4-regular-planar}]
Suppose $G$ is 4-regular and planar. By Lemma \ref{lem:CF-regions-in-4regular-planar}, it follows that there exist at least two CF faces in the embedding of $G$. Lemma~\ref{lem:CF-trianle-in-4regular-planar} implies that if one of the CF faces is a triangle, then $G$ has a CF-cut unless $G$ is the octahedron graph. 

The remaining cases are when the CF faces are not triangles. We argue that there exists a CF-cut in these cases. Let $F = v_0v_1\ldots v_{\ell-1} v_0$ be a CF face, where $\ell \geq 4$. We can separate the face $F$ from the rest of the graph $G$, unless (i) all vertices of $G$ are on the face $F$, or  (ii) there exists a vertex $x$ and edges $v_ix$ and $v_jx$ such that the edges $v_ix$ and $v_jx$ are conflicting with each other. In  the former case, there exists a CF-cut because of Lemma \ref{lem:no_edges_on_CF_face}. In the latter case, if $v_i$ and $v_j$ are not adjacent in $F$, then Lemma \ref{lem:protect_from_far} implies that there is a CF-cut in $G$. If $v_i$ and $v_j$ are adjacent in $F$, Lemma \ref{lem:CFFPT_case} implies that there is a CF-cut in this case as well. 
\qed    
\end{proof}

\input{7-4reg-proof}

\section{CF-cut for Planar Graphs}
\label{sec:planar_CF-cut}

In this section, we show that the CF-cut problem is NP-complete when $G$ is planar, has maximum degree 5, and $\widehat{G}$ is 1-regular.

\begin{theorem}\label{thm:planarNPC} Given a planar graph $G$ with maximum degree 5 and 1-regular conflict graph $\widehat{G}$, it is NP-complete to decide if there is a CF-cut in $G$.
\end{theorem}

The NP-completeness of CF-cut follows from a reduction from the Clean 3-SAT problem, which was shown to be NP-complete \cite{cygan2017hitting}.

\begin{definition}[Clean 3-SAT] A variant of $3$-SAT in which each variable appears exactly thrice, at least once positive and at least once negative. Each clause in a clean $3$-SAT formula can have two or three literals. A variable appears in a clause at most once. Moreover, by renaming a literal, we may assume that each variable in a given clean $3$-SAT instance appears twice in positive form and once in negated form. 
\end{definition}


It is straightforward to see that the CF-cut problem is in NP. We prove that Clean 3-SAT reduces to the CF-cut problem, where $G$ is a planar graph with maximum degree 5 and a 1-regular conflict graph $\widehat{G}$. Our reduction is based on the reduction used by \cite{rauch2025conflict}, but we use gadgets that are planar and uncuttable. 

In the first step of this reduction, we reduce a given instance of clean 3-SAT to a multigraph $G'$ with 1-regular conflict. We then construct a planar uncuttable gadget of maximum degree 5 with 1-regular conflict. Finally, we convert $G'$ into a simple planar graph with maximum degree 5 and 1-regular conflicts using the uncuttable gadget. The whole reduction takes $O(mn)$ steps, where $m$ is the number of clauses and $n$ is the number of variables in the clean 3-SAT instance.

\subsection{Reducing Clean 3-SAT to a multigraph $G'$ with 1-regular conflicts} \label{subsec:G'NP}

Suppose we are given a clean 3-SAT formula $\mathcal{F}$ with $n$ variables and $m$ clauses. Let the variables be $(x_1, x_2, \dots, x_n)$ and the clauses be $(C_1, C_2, \dots, C_m)$. 
We use $\mathcal{F} =  (x_1 \vee \bar{x}_3) \wedge (\bar{x}_2 \vee x_3) \wedge (\bar{x}_1 \vee x_2 \vee x_3) \wedge (x_1\vee x_2)$ as an example for illustration. The variable gadget we use in this reduction is shown in Figure~\ref{fig:var-gadget}.

To construct $G'$ from $\mathcal{F}$, we start with two special vertices $s$ and $t$. For every clause $C \in \mathcal{F}$, we add an $s-t$ path $P_C$ of length $|C|$ in ${G}'$. In the path $P_C$, when seen from the vertex $s$, the $j$th edge represents the $j$th literal in $C$. Now, for every pair of distinct edges $e, f \in E(P_C)$, we add an edge $e'$ parallel to $e$ and an edge $f'$ parallel to $f$, and we add a conflict between $e'$ and $f'$.

Now we modify each path $P_c$ as follows to implicitly create a variable gadget for each variable $x_i$ in $\mathcal F$:  we use $u_{i,k} v_{i,k}$, for $1 \leq i \leq n$ and $1 \leq k \leq 3$, to represent an edge $e \in E(P_C)$ corresponding to the $k^{th}$ occurrence of a variable $x_i$ in some clause $C \in \mathcal{F}$. We emphasize that all three occurrences of $x_i$ appear in different clauses $C \in \mathcal{F}$. 
Now, for every negative literal of $x_i$ in $C \in \mathcal F$, we replace its corresponding edge $uv \in P_C$ with a pair of parallel edges, $u_{i,2} v_{i,2}$. Of these parallel edges, one conflicts with the edge $u_{i,1} v_{i,1}$ and another conflicts with $u_{i,3} v_{i,3}$. The graph $G'$ is a planar multi-graph with maximum degree $4m$, as shown in Figure \ref{fig:G'_Complete}. This completes the construction.

\begin{figure}[h!]
\centering
\makebox[\textwidth]{%
\begin{minipage}[b]{0.35\textwidth}
\centering
\begin{tikzpicture}[scale=0.9, every node/.style={font=\normalsize}]
  \tikzstyle{vertex}=[circle,fill=black,inner sep=1.5pt]

  \node[vertex,label=left: $u_{i,1}$] (u1) at (0,0) {};
  \node[vertex,label=right: $v_{i,1}$] (v1) at (2,0) {};
  \draw (u1) -- (v1);
  \path (u1) -- (v1) coordinate[pos=0.5] (mid1); 

  \node[vertex,label=left: $u_{i,2}$] (u2) at (0,-1) {};
  \node[vertex,label=right: $v_{i,2}$] (v2) at (2,-1) {};
  \draw (u2) .. controls (0.8,-0.5) and (1.2,-0.5) .. (v2); 
  \draw (u2) .. controls (0.8,-1.5) and (1.2,-1.5) .. (v2); 
  \path (u2) .. controls (0.8,-0.5) and (1.2,-0.5) .. (v2) coordinate[pos=0.5] (midA);
  \path (u2) .. controls (0.8,-1.5) and (1.2,-1.5) .. (v2) coordinate[pos=0.5] (midB);

  \node[vertex,label=left:\normalsize $u_{i,3}$] (u3) at (0,-2) {};
  \node[vertex,label=right: $v_{i,3}$] (v3) at (2,-2) {};
  \draw (u3) -- (v3);
  \path (u3) -- (v3) coordinate[pos=0.5] (mid3); 

  \draw[dashed] (mid1) -- (midA);
  \draw[dashed] (mid3) -- (midB);

\end{tikzpicture}
\captionsetup{singlelinecheck=false}
    \caption{\small The variable gadget representing the three literals of the variable $x_i$. The parallel edges $u_{i,2}v_{i,2}$ represent the negated literal and the edges $u_{i,1}v_{i,1}$, $u_{i,3}v_{i,3}$ represent the positive literals. Dashed edges represent conflicts.}
    \label{fig:var-gadget}

\end{minipage}

\hspace{0.2in}

\begin{minipage}[b]{0.50\textwidth}
\begin{tikzpicture}[scale=0.7, every node/.style={font=\small}]
  \tikzstyle{vertex}=[circle,fill=black,inner sep=1.5pt]

  \node[vertex] (ui1) at (0,0) {};
  \node[vertex] (vi1) at (2,0) {};
  \draw (ui1) .. controls (0.7,0.3) and (1.3,0.3) .. (vi1)
       coordinate[pos=0.5] (mid-ui1vi1-top);
  \draw (ui1) .. controls (0.7,-0.3) and (1.3,-0.3) .. (vi1)
       coordinate[pos=0.5] (mid-ui1vi1-bot);

  \node[vertex] (uk1) at (4,0) {};
  \node[vertex] (vk1) at (6,0) {};
  \draw (uk1) .. controls (4.7,0.4) and (5.3,0.4) .. (vk1)
       coordinate[pos=0.5] (mid-uk1vk1-top);
  \draw (uk1) .. controls (4.7,0.0) and (5.3,0.0) .. (vk1)
       coordinate[pos=0.6] (mid-uk1vk1-mid);
  \draw (uk1) .. controls (4.7,-0.4) and (5.3,-0.4) .. (vk1)
       coordinate[pos=0.5] (mid-uk1vk1-bot);

  \node[vertex] (uj1) at (2,-1.5) {};
  \node[vertex] (vj1) at (4,-1.5) {};
  \draw (uj1) .. controls (2.7,-1.1) and (3.3,-1.1) .. (vj1)
       coordinate[pos=0.65] (mid-uj1vj1-top);
  \draw (uj1) .. controls (2.7,-1.5) and (3.3,-1.5) .. (vj1)
       coordinate[pos=0.75] (mid-uj1vj1-mid);
  \draw (uj1) .. controls (2.7,-1.9) and (3.3,-1.9) .. (vj1)
       coordinate[pos=0.5] (mid-uj1vj1-bot);

  \node[vertex] (vk2) at (6,-1.5) {};
  \draw (vj1) .. controls (4.7,-1.1) and (5.3,-1.1) .. (vk2)
       coordinate[pos=0.6] (mid-vj1vk2-top);
  \draw (vj1) .. controls (4.7,-1.9) and (5.3,-1.9) .. (vk2)
       coordinate[pos=0.3] (mid-vj1vk2-bot);

  \node[vertex] (ui2) at (0,-3) {};
  \node[vertex] (vi2) at (2,-3) {};
  \draw (ui2) .. controls (0.7,-2.5) and (1.3,-2.5) .. (vi2)
       coordinate[pos=0.5] (mid-ui2vi2-top);
  \draw (ui2) .. controls (0.7,-2.8) and (1.3,-2.8) .. (vi2)
       coordinate[pos=0.7] (mid-ui2vi2-st);
  \draw (ui2) .. controls (0.7,-3.2) and (1.3,-3.2) .. (vi2)
       coordinate[pos=0.7] (mid-ui2vi2-sb);
  \draw (ui2) .. controls (0.7,-3.5) and (1.3,-3.5) .. (vi2)
       coordinate[pos=0.5] (mid-ui2vi2-bot);

  \node[vertex] (vj2) at (4,-3) {};
  \draw (vi2) .. controls (2.7,-2.6) and (3.3,-2.6) .. (vj2)
       coordinate[pos=0.65] (mid-vi2vj2-top);
  \draw (vi2) .. controls (2.7,-3.0) and (3.3,-3.0) .. (vj2)
       coordinate[pos=0.3] (mid-vi2vj2-mid);
  \draw (vi2) .. controls (2.7,-3.4) and (3.3,-3.4) .. (vj2)
       coordinate[pos=0.7] (mid-vi2vj2-bot);

  \node[vertex] (vk3) at (6,-3) {};
  \draw (vj2) .. controls (4.7,-2.6) and (5.3,-2.6) .. (vk3)
       coordinate[pos=0.5] (mid-vj2vk3-top);
  \draw (vj2) .. controls (4.7,-3.0) and (5.3,-3.0) .. (vk3)
       coordinate[pos=0.3] (mid-vj2vk3-mid);
  \draw (vj2) .. controls (4.7,-3.4) and (5.3,-3.4) .. (vk3)
       coordinate[pos=0.5] (mid-vj2vk3-bot);

  \node[vertex] (ui3) at (0,-4.5) {};
  \node[vertex] (vi3) at (2,-4.5) {};
  \draw (ui3) .. controls (0.7,-4.1) and (1.3,-4.1) .. (vi3)
       coordinate[pos=0.5] (mid-ui3vi3-top);
  \draw (ui3) .. controls (0.7,-4.9) and (1.3,-4.9) .. (vi3)
       coordinate[pos=0.5] (mid-ui3vi3-bot);

  \node[vertex] (vj3) at (4,-4.5) {};
  \draw (vi3) .. controls (2.7,-4.1) and (3.3,-4.1) .. (vj3)
       coordinate[pos=0.5] (mid-vi3vj3-top);
  \draw (vi3) .. controls (2.7,-4.9) and (3.3,-4.9) .. (vj3)
       coordinate[pos=0.5] (mid-vi3vj3-bot);

  \draw[dashed,bend left=20] (mid-ui1vi1-top) to (mid-uk1vk1-top);
  \draw[dashed,bend left=0] (mid-ui1vi1-bot) to (mid-ui2vi2-top);
  \draw[dashed,bend left=0] (mid-ui2vi2-bot) to (mid-ui3vi3-top);
  \draw[dashed,bend right=20] (mid-ui3vi3-bot) to (mid-vi3vj3-bot);
  \draw[dashed,bend left=0] (mid-vi3vj3-top) to (mid-uj1vj1-bot);
  \draw[dashed,bend left=0] (mid-uj1vj1-top) to (mid-vi2vj2-top);
  \draw[dashed,bend left=20] (mid-ui2vi2-st)  to (mid-vi2vj2-mid);
  \draw[dashed,bend right=20] (mid-ui2vi2-sb)  to (mid-vj2vk3-bot);
  \draw[dashed,bend right=20] (mid-vi2vj2-bot) to (mid-vj2vk3-mid);
  \draw[dashed,bend left=0] (mid-vj2vk3-top) to (mid-uk1vk1-bot);
  \draw[dashed,bend left=0] (mid-uk1vk1-mid) to (mid-vj1vk2-top);
  \draw[dashed,bend left=20] (mid-vj1vk2-bot) to (mid-uj1vj1-mid);

  \draw[line width=0.9pt,rounded corners=4pt] (1.6,0.45) rectangle (4.4,-0.45); 
  \draw[line width=0.9pt,rounded corners=4pt] (-0.6,0.45) rectangle (0.6,-4.95); 
  \draw[line width=0.9pt,rounded corners=4pt] (5.4,0.45) rectangle (6.6,-4.95);  
  \draw[line width=0.9pt,rounded corners=3pt] (-0.2,-1.8) rectangle (2.3,-1.2); 
  \draw[line width=0.9pt,rounded corners=3pt] (3.6,-4.8) rectangle (6.3,-4.2); 

  \node[left=0.1cm] at (-0.6,0) { ${x_1, \bar{x}_3}$};
  \node[left=0.1cm] at (-0.6,-1.5) { ${\bar{x}_2, x_3}$};
  \node[left=0.1cm] at (-0.6,-3) {${\bar{x}_1, x_2, x_3}$};
  \node[left=0.1cm] at (-0.6,-4.5) { ${x_1, x_2}$};

  \node[above=0.3cm] at (0,0.47) { $s$};
  \node[above=0.3cm] at (1,0.45) { $x_1$};
  \node[above=0.3cm] at (3,0.45) { $x_2$};
  \node[above=0.3cm] at (5,0.45) { $x_3$};
  \node[above=0.3cm] at (6,0.47) { $t$};

\end{tikzpicture}
\captionsetup{singlelinecheck=false}
\caption{\small Completed $G'$ for $\mathcal{F} =  (x_1 \vee \bar{x}_3) \wedge (\bar{x}_2 \vee x_3) \wedge (\bar{x}_1 \vee x_2 \vee x_3) \wedge (x_1\vee x_2)$. Vertices inside overlapped rectangles are identified. Dashed edges represent conflicts.}
    \label{fig:G'_Complete}
    \end{minipage}
}
\vspace{-0.4in}
\end{figure}

\begin{lemma}
\label{lem:stcut}
If there is a CF-cut in the multi-graph $G'$ in the above construction then it is necessarily an $s-t$ cut.
\end{lemma}

\begin{proof} According to the construction of $G'$, there is an $s-t$ path $P_C$ of length $|C|$ for each clause $C \in \mathcal{F}$. Since each clause $C \in \mathcal{F}$ has two or three literals, each $s-t$ path has one or two intermediate vertices. We consider the case when $C$ has three literals and $P_C$ has two intermediate vertices. Let the intermediate vertices be $u$ and $v$. 
In this case, edges $su$, $uv$, and $vt$ will have two copies of parallel edges $(su)'$, $(uv)'$, and $(vt)'$ respectively. Of the two copies of $(su)'$, one conflicts with a copy of $uv'$ and another with a copy of $(vt)'$. The remaining copy of $(uv)'$ conflicts with the remaining copy of $(vt)'$. For the sake of contradiction, let us assume that there is a CF-cut $M = E(A, B)$ such that both $s$ and $t$ belong to the same set, say $B$, and $u, v \in A$. Then we have two conflicting edges $(su)', (vt)' \in M$, a contradiction. If $u \in A$ and $v \in B$, then there are two conflicting edges $(su)', (uv)' \in M$, again a contradiction.
If $u \in B$ and $v \in A$, then we have two conflicting edges $(uv)'$ and $(vt)'$ in $M$, which is a contradiction.
Therefore, all CF-cuts in $G'$ are $s-t$ cuts. This completes the proof for the case when $C$ has three literals. The case when $C$ has two literals can be reasoned similarly.
\qed
\end{proof}

\begin{lemma} \label{lem:G'NP} The multi-graph $G'$ constructed above has a CF-cut if and only if its corresponding clean 3-SAT formula $\mathcal{F}$ is satisfiable.
\end{lemma}

\begin{proof} For every edge $e \in P_C$, let $S(e)$ denote a multi-set of edges containing $e$ and all its copy edges $e'$. The conflicts assigned in $G'$ ensure the following: whenever $S(e)$ corresponding to a positive literal $x_i$ becomes part of a CF-cut, the edges corresponding to the literal $\bar{x_i}$ do not, and vice versa. For one direction, let us assume that formula $\mathcal F$ is satisfiable and $\alpha$ is a satisfying assignment for $\mathcal F$. 
Each clause contains a true literal in $\alpha$, and we select the edges corresponding to one of the true literals to be in the cut. This ensures that $s$ and $t$ are separated. 



For the other direction, if $M$ is a CF-cut in $G'$, by Lemma \ref{lem:stcut}, $M$ must be an $s-t$ cut. Hence, for all clauses $C$ in $\mathcal{F}$, $M$ must contain the edges $S(e)$ for some $e \in P_C$. By setting the literals corresponding to the chosen edges to True, we recover a satisfying assignment to $\mathcal F$. The variable gadget ensures we cannot pick inconsistent assignments to any variable.
\qed
\end{proof}

\subsection{Constructing an uncuttable planar gadget}
This section explains how to obtain an uncuttable planar graph of maximum degree 5 and 1-regular conflicts from the square of an even cycle. 

We will use this gadget to convert the multigraph $G'$ obtained in the previous section into a simple planar graph with maximum degree 5.

\begin{lemma} \label{lem:singcut} Let $H_{2n}$ be the square of a cycle graph $C_{2n} \; (n \geq 3)$ and $\{0, \allowbreak 1, \allowbreak 2, \allowbreak \ldots, \allowbreak 2n-1\}$ be its vertices when read clockwise. Suppose we assign a conflict between edges $\{i, (i+1)\} \bmod 2n $ and $\{i, (i+2)\} \bmod 2n$ for $0 \leq i \leq 2n-1$. The only CF-cut that $H_{2n}$ possesses is the set of edges of the form $\{i, (i+1)\} \bmod 2n$ for $0 \leq i \leq 2n-1$.
\end{lemma}

\begin{proof} As we can observe from Figure~\ref{fig:Square_C_2n}, in the square of every cycle graph $C_{2n} \; (n \geq 3)$, there are only two types of edges: $\{i, (i+1)\} \bmod 2n $ and $\{i, (i+2)\} \bmod 2n$ for $0 \leq i \leq 2n-1$. We can easily verify from Figure~\ref{fig:Square_C_2n} that by deleting all the edges of the form $\{i, (i+1) \} \bmod 2n$ for $0 \leq i \leq 2n-1$, the graph becomes disconnected and forms two cycle graphs, each being $C_n$: one on all odd vertices and another on all even vertices. Moreover, the edges $\{i, (i+1)\} \bmod 2n$ for $0 \leq i \leq 2n-1$, that form the cut, are conflict-free. We refer to this cut as~$M$.

To prove that $M$ is the only CF-cut in $H_{2n}$, 
let us consider an arbitrary CF-cut, say $M'$. There must be some vertex $j$, $0\leq j \leq 2n-1$, such that $j$ and $(j+1) \bmod 2n$ are on opposite sides of the cut. 

Let $M'$ divide the vertices into parts $A$ and $B$. We may assume that the vertex $j$ belongs to $A$, and the vertex $(j+1) \bmod 2n$ belongs to $B$. Since $\{j, (j+1) \} \bmod 2n $ and $\{j, (j+2) \}\bmod 2n$ for are conflicting, the vertex $(j+2) \bmod 2n$ cannot belong to the set $B$. This implies that $(j+2) \bmod 2n$ must be in $A$. Thus the edge $\{(j+1), (j+2)\} \bmod 2n$ must belong to the cut $M'$. Therefore, if $\{ i, (i+1)\} \bmod 2n$ belongs to the CF-cut $M'$, then the edge $\{(i+1) , (i+2) \}\bmod 2n$ also belong to $M'$ for $0 \leq i \leq 2n-1$. This implies that $M'$  contains all the edges in the cut $M$. For any edge outside $M$, it can be seen that there is an edge in $M$ that conflicts with it. Since $M'$ is conflict-free, this implies that $M' = M$.
\qed
\end{proof}

Now we modify the graph in Figure~\ref{fig:Square_C_2n} so that it becomes uncuttable. Below we discuss such a construction.

Let $H_{2n}$ be the square of $C_{2n} \; (n \geq 6)$ with 1-regular conflicts as shown in Figure~\ref{fig:Square_C_2n}. The vertices of even parity lie on the outer face of $H_{2n}$. We modify $H_{2n}$ into an uncuttable planar graph \(\widetilde{H}_{2n}\) as follows: we start with a vertex on the outer face of $H_{2n}$, say $i$ and consider its four immediate successors $i+1, i+2, i+3, i+4$. Similarly, we consider another vertex on the outer face $j$ such that $\{j, j+1, j+2, j+3, j+4\}$ is disjoint from 
$\{i, i+1, i+2, i+3, i+4\}$. All vertices are considered modulo $2n$.
After this, we introduce four new vertices, namely, $i', (i+2)', j'$ and $(j+2)'$ that split the edges $\{i, i+2\}, \{i+2, i+4\}, \{j, j+2\}$ and $\{j+2, j+4\}$ respectively. We then add six new edges: $\{i+1, i'\}, \{i+3, (i+2)'\},
 \{i', (i+2)'\}, \{j+1, j'\}, \{j+3, (j+2)'\}$ and $\{j', (j+2)'\}$.
We establish a 1-regular conflict between these newly introduced $10$ edges (four edges as a result of edge splits and six edges drawn separately as mentioned before) as depicted in Figure~\ref{fig:modified_C_2n}. This completes the construction of \(\widetilde{H}_{2n}\).

\begin{figure}[h!]

\vspace{-0.2in}
\captionsetup[subfigure]{skip=0pt}
\begin{subfigure}[b]{0.4\textwidth}
\begin{tikzpicture}[scale=1.8, every node/.style={font=\small}, xscale=-1]

\tikzstyle{vertex} = [circle, fill=black, inner sep=1.5pt]

  \def\Rout{1.2}   
  \def\Rin{0.7}    

  \foreach \i in {0,...,11} {
    \ifodd\i
      \node[vertex] (v\i) at ({\Rin*cos(360/12*\i)}, {\Rin*sin(360/12*\i)}) {};
    \else
      \node[vertex] (v\i) at ({\Rout*cos(360/12*\i)}, {\Rout*sin(360/12*\i)}) {};
    \fi
    \pgfmathsetmacro{\xoff}{\i == 6 || \i == 10 || \i == 0 ? 0.20 : 0.10}
    \pgfmathsetmacro{\xoff}{\i == 3 || \i == 9 ? 0 : \xoff}
    \pgfmathsetmacro{\yoff}{\i == 6 || \i == 10 || \i == 0 ? 0.0 : 0.15}

    \node at ($(v\i)+(\xoff,\yoff)$) {\small \i};

  }

  \foreach \i in {0,...,11} {
    \pgfmathtruncatemacro{\j}{mod(\i+1,12)}
    \pgfmathtruncatemacro{\k}{mod(\i+2,12)}
    \draw (v\i) -- (v\j);
    \draw (v\i) -- (v\k);
  }

  \draw[dashed, bend left=40] ($(v0)!0.7!(v1)$) to ($(v0)!0.4!(v2)$);
  \draw[dashed, bend right=40] ($(v1)!0.4!(v2)$) to ($(v1)!0.4!(v3)$);
  \draw[dashed, bend left=30] ($(v2)!0.6!(v3)$) to ($(v2)!0.4!(v4)$);
  \draw[dashed, bend right=30] ($(v3)!0.4!(v4)$) to ($(v3)!0.4!(v5)$);
  \draw[dashed, bend left=40] ($(v4)!0.6!(v5)$) to ($(v4)!0.4!(v6)$);
  \draw[dashed, bend right=30] ($(v5)!0.4!(v6)$) to ($(v5)!0.4!(v7)$);
  \draw[dashed, bend left=30] ($(v6)!0.6!(v7)$) to ($(v6)!0.4!(v8)$);
  \draw[dashed, bend right=30] ($(v7)!0.4!(v8)$) to ($(v7)!0.4!(v9)$);
  \draw[dashed, bend left=30] ($(v8)!0.5!(v9)$) to ($(v8)!0.4!(v10)$);
  \draw[dashed, bend right=40] ($(v9)!0.4!(v10)$) to ($(v9)!0.4!(v11)$);
  \draw[dashed, bend left=40] ($(v10)!0.6!(v11)$) to ($(v10)!0.35!(v0)$);
  \draw[dashed, bend right=30] ($(v11)!0.45!(v0)$) to ($(v11)!0.4!(v1)$);

\end{tikzpicture}
\caption{}
\label{fig:Square_C_2n}
\vspace{0.1in}

\end{subfigure}
\hfill
\begin{subfigure}[b]{0.4\textwidth}
\centering

\begin{tikzpicture}[scale=1.8, every node/.style={font=\small}, xscale=-1]
\tikzstyle{vertex} = [circle, fill=black, inner sep=1.5pt]

\def\Rout{1.2}   
\def\Rin{0.7}    

\foreach \i in {0,...,11} {
  \ifodd\i
    \node[vertex] (v\i) at ({\Rin*cos(360/12*\i)}, {\Rin*sin(360/12*\i)}) {};
  \else
    \node[vertex] (v\i) at ({\Rout*cos(360/12*\i)}, {\Rout*sin(360/12*\i)}) {};
  \fi
}

\node at ($(v0)+(0,0.12)$) {\small $i$};
\node at ($(v1)+(-0.2,-0.1)$) {\small $i+1$};
\node at ($(v2)+(0,0.12)$) {\small $i+2$};
\node at ($(v3)+(0,-0.19)$) {\small $i+3$};
\node at ($(v4)+(0,0.12)$) {\small $i+4$};

\node at ($(v6)+(0,0.12)$) {\small $j$};
\node at ($(v7)+(0.23,0.1)$) {\small $j+1$};
\node at ($(v8)+(0,-0.12)$) {\small $j+2$};
\node at ($(v9)+(0,0.15)$) {\small $j+3$};
\node at ($(v10)+(0,-0.12)$) {\small $j+4$};

\foreach \i in {0,...,11} {
  \pgfmathtruncatemacro{\j}{mod(\i+1,12)}
  \pgfmathtruncatemacro{\k}{mod(\i+2,12)}
  \draw (v\i) -- (v\j);
  \draw (v\i) -- (v\k);
}

\node[vertex,label=left:$i'$]   (vi')   at ($(v0)!0.5!(v2)$) {};
\node[vertex,label=above:$(i+2)'$] (vi2')  at ($(v2)!0.5!(v4)$) {};
\node[vertex,label=right:$j'$]    (vj')   at ($(v6)!0.5!(v8)$) {};
\node[vertex,label={[label distance=0.3cm]below:$(j+2)'$}] (vj2')  at ($(v8)!0.5!(v10)$) {};

\draw (v1) -- (vi');
\draw (v3) -- (vi2');
\draw (vi') .. controls ($(vi') + (0,0.8)$) and ($(vi2') + (1,0.6)$) .. node[pos=0.15, coordinate] (m_i_ip2p) {} (vi2');
\draw (v7) -- (vj');
\draw (v9) -- (vj2');
\draw (vj') .. controls ($(vj') + (0,-0.8)$) and ($(vj2') + (-1,-0.6)$) .. node[pos=0.15, coordinate] (m_jp_j2p) {} (vj2');

\coordinate (m_i_ip2)   at ($(vi')!0.5!(v2)$);     
\coordinate (m_i2p_i3)  at ($(vi2')!0.4!(v3)$);
\coordinate (m_i2p_i4)  at ($(vi2')!0.3!(v4)$);
\coordinate (m_j2p_j3)  at ($(vj2')!0.4!(v9)$);
\coordinate (m_j2p_j4)  at ($(vj2')!0.3!(v10)$);   
\coordinate (m_jp_j2)   at ($(vj')!0.5!(v8)$);     
\coordinate (m_i1_ip)   at ($(v1)!0.5!(vi')$);     
\coordinate (m_j1_jp)   at ($(v7)!0.5!(vj')$);     

\draw[dashed, bend left=35]  (m_i_ip2)  to (m_i_ip2p);   
\draw[dashed, bend left=25] (m_i2p_i3) to (m_i2p_i4);   
\draw[dashed, bend left=25] (m_j2p_j3) to (m_j2p_j4);   
\draw[dashed, bend left=35]  (m_jp_j2)  to (m_jp_j2p);   
\draw[dashed, bend right=25] (m_i1_ip)  to (m_j1_jp);    

\draw[dashed, bend left=40]  ($(v0)!0.55!(v1)$) to ($(v0)!0.3!(v2)$);
\draw[dashed, bend right=40] ($(v1)!0.35!(v2)$) to ($(v1)!0.35!(v3)$);
\draw[dashed, bend left=30]  ($(v2)!0.45!(v3)$) to ($(v2)!0.25!(v4)$);
\draw[dashed, bend right=30] ($(v3)!0.35!(v4)$) to ($(v3)!0.35!(v5)$);
\draw[dashed, bend left=40]  ($(v4)!0.5!(v5)$) to ($(v4)!0.3!(v6)$);
\draw[dashed, bend right=30] ($(v5)!0.35!(v6)$) to ($(v5)!0.35!(v7)$);
\draw[dashed, bend left=30]  ($(v6)!0.55!(v7)$) to ($(v6)!0.3!(v8)$);
\draw[dashed, bend right=30] ($(v7)!0.35!(v8)$) to ($(v7)!0.35!(v9)$);
\draw[dashed, bend left=30]  ($(v8)!0.45!(v9)$) to ($(v8)!0.25!(v10)$);
\draw[dashed, bend right=40] ($(v9)!0.35!(v10)$) to ($(v9)!0.35!(v11)$);
\draw[dashed, bend left=40]  ($(v10)!0.55!(v11)$) to ($(v10)!0.3!(v0)$);
\draw[dashed, bend right=30] ($(v11)!0.35!(v0)$) to ($(v11)!0.35!(v1)$);

\end{tikzpicture}

\vspace{-0.1in}
\caption{}
\label{fig:modified_C_2n}
\vspace{-0.1in}
\end{subfigure}
\caption{\small (a) Square of $C_{2n}$ ($n=6$) with $1$-regular conflict. Dashed edges represent conflicts. (b) Square of $C_{2n}$ when modified into \(\widetilde{H}_{2n}\).} 


\label{fig:Uncuttable_planar}
\vspace{-0.2in}
\end{figure}
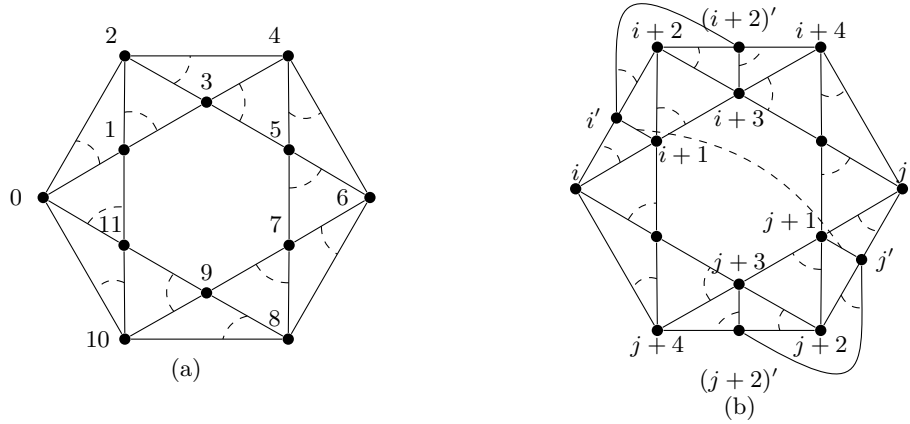

We need the following lemma to complete the proof of Theorem \ref{thm:planarNPC}.

\begin{lemma}\label{lem:planaruncuttable}
For $n \geq 6$, the graph \(\widetilde{H}_{2n}\) constructed above is an uncuttable planar graph of max degree 5.
\end{lemma}

\begin{proof}
It is clear from Figure~\ref{fig:modified_C_2n} that the graph \(\widetilde{H}_{2n}\) has no crossing edges and is hence planar. For all values of $n \geq 6$, the graph has four vertices of degree $5$, namely $i+1, i+3, j+1$ and $j+3$, and the remaining vertices of degree $4$. 
Now we prove that the graph \(\widetilde{H}_{2n}\) is uncuttable.  
As mentioned in Lemma \ref{lem:singcut}, the graph $H_{2n}$ of Figure~\ref{fig:Square_C_2n} has exactly one CF-cut. We will show that the modifications resulting in \(\widetilde{H}_{2n}\) will eliminate this cut.

Now we use the following two cases to prove our claim:
\paragraph{Case 1 (The edge $\{(i+1) , i'\}$ is part of the CF-cut):} Let the cut be $M = (A, B)$. We assume that the vertex $i+1$ belongs to one set, say, $A$, and $i'$ belongs to another set, say, $B$. Now, the following two sub-cases arise. 

\paragraph{Sub-Case 1.1 ($i+2 \in A$):} Let $(i+2) \in A$. This means that the edge $\{(i+1), (i+2)\}$ does not belong to the CF-cut $M$ but the edge $\{i' , (i+2)\}$ does. This implies that the edge $\{i', (i+2)' \}) \notin M$ and the vertex $(i+2)' \in B$. Since, edges $\{(i+2) , (i+3) \}$ and $\{(i+2), (i+2)' \}$ are conflicting, $(i+3)  \in A$. Similarly, the conflict between the edges $\{(i+2)', (i+3) \}$ and $\{(i+2)', (i+4) \}$ implies that vertex $(i+4)  \in B$ and the edge $\{(i+3) , (i+4) \} \in M$. At this point, we can propagate this using an argument similar to the one in the proof of Lemma \ref{lem:singcut}, and infer that $\{j, (j+1)\} \in M$. Since the edges $\{j, j'\}$ and $\{j, j+1\}$ are conflicting, $j'$ is in the same side of the cut as $j$ and the edge $\{j', (j+1)\}\in M$. This contradicts our assumption of Case 1 that the edge $\{(i+1),i'\} \in M$. 

\paragraph{Sub-Case 1.2 ($(i+2) \in B$):} In this case, the edge $\{i', (i+2) \}$ does not belong to the CF-cut $M$ but $\{(i+1), (i+2) \}$ does. This implies that the edge $\{(i+1) , (i+3) \} \notin M$ and the vertex $(i+3) \in A$. Since the edges $\{(i+2), (i+3) \}$ and $\{(i+2), (i+2)' \}$ are conflicting, it follows that $(i+2)' \in B$. Similarly, the conflict between the edges $\{(i+2)', (i+3) \}$ and $\{(i+2)', (i+4) \}$ implies $(i+4) \in B$. Thus the edge $\{(i+3), (i+4) \} \in M$. From this point, we can reason in the same manner as in Subcase 1.1, and reach a contradiction where two conflicting edges $\{(i+1), i'\}$, and $\{j', (j+1) \}$ are both found to be in $M$.

\paragraph{Case 2 (The edge $\{(i+1), i'\}$ is not in the CF-cut):} Let both the vertices $(i+1)$ and $i'$ belong to $A$. Since the edges $\{i, i'\}$ and $\{i, (i+1)\}$ are conflicting, the vertex $i \in A$. We can use the argument moving in an anti-clockwise manner to infer that vertices $i-1, i-2, \ldots$ are in $A$. We can thus infer that vertices $(j+4), (j+3), (j+2)', (j+2), (j+1) $ are all in $A$. Since the edges $\{j', (j+2) \}$ and $\{j', (j+2)' \}$ conflict, it follows that $j'$ is in $A$ as well. Now due to the conflicts, it follows that $j, j-1, j-2, \ldots $ are all in $A$. It follows that all the vertices are in $A$ and hence there is no cut in this case. 

From the above arguments, we can say that the graph \(\widetilde{H}_{2n}\) as shown in Figure~\ref{fig:Uncuttable_planar} is an uncuttable, planar graph of maximum degree 5. This completes the proof. 
\qed 
\end{proof}



Now we are ready to prove Theorem \ref{thm:planarNPC}.

\begin{proof}[Proof of Theorem \ref{thm:planarNPC}.] Recall that we have shown that the multigraph $G'$ has an $s-t$ cut if and only if the formula $\mathcal F$ is satisfiable. We need to obtain a planar simple graph from $G'$. We replace each vertex of $G'$ with a sufficiently large uncuttable planar graph \(\widetilde{H}_{2t}\). Since the maximum degree of $G'$ is $4m$ (where $m$ is the number of clauses in $\mathcal F$), setting $t = 4m$ suffices.
The adjacencies of a vertex $v \in G'$ are distributed among the vertices on the outer face of \(\widetilde{H}_{2t}\) in such a way that each vertex on the outer face of \(\widetilde{H}_{2t}\) is incident with at most one edge of $v$. Note that all the vertices on the outer face of \(\widetilde{H}_{2t}\) are of degree 4. 
Thus we get an uncuttable, planar graph with 1-regular conflicts and maximum degree 5. 
The graph $G'$ has at most $3n-1$ vertices and at most $10n$ edges. 
The gadget \(\widetilde{H}_{2t}\) can  be constructed in polynomial time. This completes the reduction from a clean 3-SAT instance to an uncuttable, simple, planar graph.
\qed
\end{proof}


\subsection{NP-completeness for maximum degree beyond 5}

Though it is not surprising, one may wonder if the NP-completeness result holds for planar graphs that have maximum degree larger than 5. We establish below that the NP-completeness result in Theorem \ref{thm:planarNPC} can be extended for planar $G$ with maximum degree greater than 5 as well.  We need the following lemma.

\begin{lemma}\label{lem:extendH}
Let \(H\) be a simple connected graph with $1$-regular conflicts. We obtain a new graph $H'$ with vertex set $V(H') = V(H) \cup \{u_1\}$
and edge set $E(H') = E(H) \cup \{u_1v_1, u_1v_2\}$, where $v_1, v_2$ are two vertices in $H$. Further we introduce conflicts between the new edges $u_1 v_1$ and $u_1v_2$.
Then $H'$
is an uncuttable, 3-degenerate graph if and only if $H$ is uncuttable and 3-degenerate.
\end{lemma}

\begin{proof}
An example of the construction is illustrated in  Figure~\ref{fig:Uncuttable_3-degenerate}. Let $H$ be an uncuttable $3$-degenerate graph with $1$-regular conflicts. By the definition of $3$-degeneracy, after adding a vertex $u_1$ of degree $2$ to $H$, the resulting graph is still $3$-degenerate. To check if \(H'\) remains uncuttable, for the sake of contradiction, consider a CF-cut $M' = E(A, B)$ of $H'$. Since \(H\) is uncuttable, all its vertices belong to one set, say $A$. The only possible candidate for the set $B$ would be the vertex $u_1$. But the edges $u_1v_1$ and $u_1v_2$  are conflicting, resulting in a contradiction.

For the other direction, if \(H'\) is $3$-degenerate then deleting a degree $2$ vertex, $u_1$, does not change its $3$-degeneracy. To prove that an uncuttable \(H'\) implies an uncuttable \(H\), for the sake of contradiction, assume that \(H\) has a CF-cut $M = E(A, B)$. There are two cases. First, suppose that both $v_1, v_2$ are on the same side, say $A$. Then the cut $M' = E(A\cup \{u_1\}, B)$ is a valid CF-cut for $H'$, which is a contradiction. If $v_1$ and $v_2$ are on different sides of the cut, then adding $u_1$ to either side results in a valid CF-cut for $H'$, again a contradiction.
\qed
\end{proof}

\begin{theorem}
     Let $D \geq 5$ be a positive integer. The CF-cut problem is NP-complete even when restricted to planar graphs $G$ with maximum degree equal to $D$ and 1-regular 
conflicts.
\end{theorem}
\begin{proof}
    The proof of Theorem \ref{thm:planarNPC} uses a reduction from Clean 3-SAT. The resulting graph $G$ is a planar graph with maximum degree 5. For $D > 5$, we can use Lemma \ref{lem:extendH} to produce graphs that have maximum degree $D$ as desired and are uncuttable  if and only if $G$ is uncuttable. By choosing $v_1, v_2$ from the same face, it can be ensured that the resulting graph remains planar. 
    \qed
\end{proof}

\section{CF-cut on 3-Degenerate Graphs}
\label{sec:3dg}

In this section, we show that the CF-Cut problem is NP-complete when $G$ is 3-degenerate and has maximum degree 5, with a 1-regular conflict graph $\widehat{G}$.




\begin{theorem}
\label{thm:3-degenerate_graph}
Given a 3-degenerate graph $G = (V, E)$ of maximum degree $5$ and $1$-regular conflict graph  $\widehat{G}$, it is NP-complete to decide if a CF-cut exists in~$G$.
\end{theorem}

The proof starts with the graph $G'$ derived in Section \ref{subsec:G'NP} but we need a different gadget as the goal is a 3-degenerate graph. 
We first observe that the graph $H$ in  Figure \ref{fig:Uncuttable_3-degenerate}(a) is 3-degenerate, has maximum degree 5 and is uncuttable for the assigned conflicts. Lemma \ref{lem:uncuttableH''} builds upon this graph to obtain a class of uncuttable 3-degenerate graphs.

\begin{figure}[h!]
\centering

\begin{minipage}{0.25\textwidth}
\centering
\begin{tikzpicture}[scale=0.8, every node/.style={font=\small}, thick]

\tikzstyle{vertex} = [circle, fill=black, inner sep=1.5pt]

\node[vertex] (1) at (-0.6,2) {};
\node[above=0.5pt of 1] { $v_1$};

\node[vertex] (2) at (0.6,2) {};
\node[above=0.5pt of 2] { $v_2$};

\node[vertex] (4) at (0,0.8) {};
\node[right=2pt of 4] { $v_4$};

\node[vertex] (3) at (-0.8,0.4) {};
\node[below=2pt of 3] { $v_3$};

\node[vertex] (5) at (0.8,0.4) {};
\node[below=2pt of 5] { $v_5$};

\node[vertex] (6) at (0,-1) {};
\node[below=2pt of 6] { $v_6$};

\draw (1) -- (4);
\draw (1) -- (3);
\draw (1) to[bend right=65] node[pos=0.25, coordinate] (mid16) {} (6);
\draw (2) -- (4);
\draw (2) to[bend left=65] node[pos=0.25, coordinate] (mid26) {} (6);
\draw (2) -- (5);
\draw (3) -- (4);
\draw (3) to[bend right=25] node[pos=0.25, coordinate] (mid35) {} (5);
\draw (3) -- (6);
\draw (4) -- (5);
\draw (4) -- (6);
\draw (5) -- (6);

\draw[dashed, bend right=20] (mid16) to ($(1)!0.5!(3)$);
\draw[dashed, bend left=20] (mid26) to ($(2)!0.5!(5)$);
\draw[dashed, bend left=20] ($(1)!0.6!(4)$) to ($(4)!0.4!(2)$);
\draw[dashed, bend left=20] ($(3)!0.5!(4)$) to (mid35);
\draw[dashed, bend right=20] ($(4)!0.55!(5)$) to ($(5)!0.3!(6)$);
\draw[dashed, bend left=20] ($(3)!0.5!(6)$) to ($(4)!0.55!(6)$);

\end{tikzpicture}
\caption*{(a)}
\end{minipage}
\hfill
\begin{minipage}{0.25\textwidth}
\centering
\begin{tikzpicture}[scale=0.8, every node/.style={font=\small}, thick]

\tikzstyle{vertex} = [circle, fill=black, inner sep=1.5pt]

\node[vertex] (u1) at (0,3) {};
\node[left=2pt of u1] { $u_1$};

\node[vertex] (1) at (-0.6,2) {};
\node[left=0.25pt of 1] {$v_1$};

\node[vertex] (2) at (0.6,2) {};
\node[right=0.25pt of 2] {$v_2$};

\node[vertex] (4) at (0,0.8) {};
\node[right=2pt of 4] { $v_4$};

\node[vertex] (3) at (-0.8,0.4) {};
\node[below=2pt of 3] { $v_3$};

\node[vertex] (5) at (0.8,0.4) {};
\node[below=2pt of 5] { $v_5$};

\node[vertex] (6) at (0,-1) {};
\node[below=2pt of 6] { $v_6$};

\draw (1) -- (4);
\draw (1) -- (3);
\draw (1) to[bend right=65] node[pos=0.25, coordinate] (mid16) {} (6);
\draw (2) -- (4);
\draw (2) to[bend left=65] node[pos=0.25, coordinate] (mid26) {} (6);
\draw (2) -- (5);
\draw (3) -- (4);
\draw (3) to[bend right=25] node[pos=0.25, coordinate] (mid35) {} (5);
\draw (3) -- (6);
\draw (4) -- (5);
\draw (4) -- (6);
\draw (5) -- (6);

\draw (u1) -- (1);
\draw (u1) -- (2);
\draw[dashed, bend right=20] ($(u1)!0.4!(1)$) to ($(u1)!0.4!(2)$);

\draw[dashed, bend right=20] (mid16) to ($(1)!0.5!(3)$);
\draw[dashed, bend left=20] (mid26) to ($(2)!0.5!(5)$);
\draw[dashed, bend left=20] ($(1)!0.6!(4)$) to ($(4)!0.4!(2)$);
\draw[dashed, bend left=20] ($(3)!0.5!(4)$) to (mid35);
\draw[dashed, bend right=20] ($(4)!0.55!(5)$) to ($(5)!0.3!(6)$);
\draw[dashed, bend left=20] ($(3)!0.5!(6)$) to ($(4)!0.55!(6)$);

\end{tikzpicture}
\caption*{(b)}
\end{minipage}

\caption{\small (a) An uncuttable, $3$-degenerate graph $H$ with $1$-regular conflicts, (b) An uncuttable, $3$-degenerate graph $H'$ derived from $H$ after adding a vertex $u_1$ and a pair of conflicting edges $u_1 v_1$ and $u_1v_2$ to $H$. Dashed edges represent conflicts.}
\label{fig:Uncuttable_3-degenerate}
\vspace{-0.28in}
\end{figure}

\begin{lemma} 
\label{lem:small-3dg}
The graph $H$ with 1-regular conflicts as shown in Figure \ref{fig:Uncuttable_3-degenerate}(a), is an uncuttable, 3-degenerate graph. 
\end{lemma}
\begin{proof} We can easily verify that the graph $H$ is 3-degenerate. The sequence $v_1, v_2, v_3, v_4, v_5, v_6$ is a degeneracy ordering. To show that the graph with the shown conflicts is uncuttable, for the sake of contradiction, let us assume that $M = E(A, B)$ is a CF-cut of \(H\). 

We first note that the edge $v_1v_6 \notin M$. Suppose not, say $v_1v_6$ is part of the cut. Without loss of generality, we may  assume that $v_1 \in A$ and $v_6 \in B$. This implies that $v_3 \in A$ and since $v_6 \in B$ and $v_3 \in A$, we can further infer that $v_4 \in B$. 
This implies that $v_5 \in B$.  This makes two conflicting edges $v_3 v_4$ and $v_3 v_5 \in M$, a contradiction.

Next, we note that the edge $v_2v_6 \notin M$. Suppose not, say $v_2v_6$ is part of the cut. Without 
loss of generality, we may assume that $v_2 \in A$ and $v_6 \in B$. This implies that $v_5 \in A$. Since $v_6 \in B$ and $v_5 \in A$ we can infer $v_4 \in A$. Similarly, $v_6 \in B$ and $v_4 \in A$ together imply $v_3 \in B$. This makes two conflicting edges $v_3 v_4$ and $v_3 v_5 \in M$, a contradiction. 

As explained above, we know that the edges $v_1v_6, v_2v_6$ are not in $M$. Without loss of generality, we assume that $v_1, v_2, v_6 \in A$. This successively implies that $v_4, v_5, v_3 \in A$. Hence all the six vertices are on the same side of the cut. This completes the proof.
\qed
\end{proof}


We build on this graph $H$ and repeatedly apply the construction in Lemma \ref{lem:extendH} to obtain uncuttable graphs of arbitrary size. This is summarized in the below lemma.

\begin{lemma}
\label{lem:uncuttableH''}
There is a family of uncuttable graphs $H_i$ with 1-regular conflicts that satisfy the following properties: 
\begin{itemize}
    \item $|V(H_i)| = 13 + 7i$, where 8 vertices are of degree 5 in $H_i$, $1+ 7i$ vertices are of degree 4, and 4 vertices are of degree 2. 
    \item $H_i$ has maximum degree 5 and $|E(H_i)| = 2 |V(H_i)|$.
    \item Repeatedly deleting degree 2 vertices from $H_{i+1}$ yields $H_i$ for $i \geq 1$, and repeatedly deleting degree 
    2 vertices from $H_1$ yields $H$ (as referred to in Lemma \ref{lem:small-3dg}).
\end{itemize}
\end{lemma}

\begin{proof}
    
Here, we discuss a constructive proof. 
Let $H$ be the simple, connected, 3-degenerate, uncuttable graph with 1-regular conflicts as shown in Figure \ref{fig:Uncuttable_3-degenerate}(a). We have $|E(H)| = 2 |V(H)|$. 
Notice that the construction in Lemma \ref{lem:extendH} results in a 3-degenerate, uncuttable graph maintaining the ratio between edges and vertices. 

Starting from $H$, we apply the construction in Lemma \ref{lem:extendH} repeatedly to obtain $H_i$, $i\geq1$. 
First we construct $H_1$ from $H$ as follows: we add three new vertices $u_1, u_2$ and $u_3$ and six new edges $u_1v_1, u_1v_2, u_2 v_1, u_2 v_3, u_3 v_2$ and $u_3v_5$ to \(H\) where the pairs 
$(u_1v_1, u_1v_2)$, $(u_2 v_1, u_2 v_3)$, $(u_3 v_2, u_3 v_5)$
conflict with each other.
At this point, $H_1$ under construction has two types of vertices: 
$v_j$ for $1 \leq j \leq 6$ of degree 5, and vertices $u_k$ for $1 \leq k \leq 3$ of degree 2. We continue to add new vertices $u_{\ell}$ for $\ell \geq 4$ and two conflicting edges $u_{\ell} u_r$ and $u_{\ell} u_s$, such that $u_r$ and $u_s$ have degrees at most 4 as shown in Figure \ref{fig:Extendable_3-degen}(a). The resulting graph containing the vertices $v_1,v_2,\ldots,v_6$ and $u_1,u_2,\ldots,u_{14}$  is $H_1$. 


We extend this construction to get $H_2$ by adding seven new vertices in two layers: we first add $u_{15}, u_{16}, u_{17}$ 
with  conflicting pairs of edges $(u_{15} u_{11}, u_{15} u_{12})$, $(u_{16}u_{12}, u_{16}u_{13})$, and $(u_{17}u_{13},u_{17}u_{14} )$, as shown in the Figure \ref{fig:Extendable_3-degen}(b).
We then add four new vertices $u_{18}, u_{19}, u_{20}$ and $u_{21}$ with the 
following conflicting pairs of edges: $(u_{18}u_{11}, u_{18}u_{15})$, $(u_{19}u_{15}, u_{19}u_{16})$, $(u_{20} u_{16}, u_{20}u_{17})$ and $(u_{21}u_{17}, u_{21}u_{14})$. 

We can repeat this process to obtain $H_i$, $i\geq 3$, of the desired size. Since the construction repeatedly uses the process described in Lemma \ref{lem:extendH}, it follows that the graphs at each stage of construction are uncuttable. 
Since we add degree two vertices at each stage, it follows that 
the graphs are 3-degenerate throughout and that $|E(H_i)| = 2|V(H_i)|$.
\qed
\end{proof}

\begin{figure}[htbp]
    \centering
    \begin{subfigure}[b]{0.4\textwidth}
        \centering
        \includegraphics[width=\textwidth]{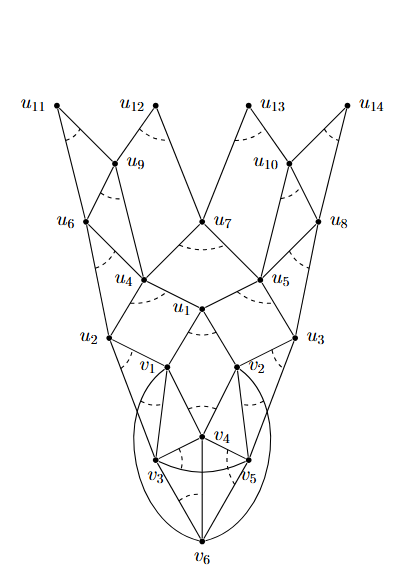}
        \caption{}
        \label{fig:Extendable_3-degen_a}
    \end{subfigure}
    \hfill
    \begin{subfigure}[b]{0.4\textwidth}
        \centering
        \includegraphics[width=\textwidth]{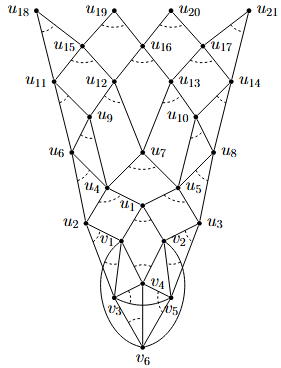}
        \caption{}
        \label{fig:Extendable_3-degen_b}
    \end{subfigure}
    \caption{\small $(a)$ An uncuttable, extendable, $3$-degenerate graph $H_i$ for $i=1$, of max degree $5$ with $1$-regular conflicts, $(b)$ An uncuttable, extendable, $3$-degenerate graph $H_i$ for $i=2$, of max degree $5$. The dotted lines represent the conflicts.}
    \label{fig:Extendable_3-degen}
\end{figure}

\noindent Now we are ready to prove Theorem \ref{thm:3-degenerate_graph}.

\begin{proof}[Proof of Theorem \ref{thm:3-degenerate_graph}] It is straightforward to see that the problem is in NP since we only need to verify that a given cut $M$ separates the graph, and that $M$ is conflict-free. We focus on showing  the NP-hardness. 
We start with the reduction from a clean 3-SAT formula $\mathcal{F}$ stated in Section \ref{subsec:G'NP}. By Lemma \ref{lem:G'NP}, it follows that the graph $G'$ obtained as a result of the reduction has a CF-cut if and only if the given clean 3-SAT instance is satisfiable. 
We need to obtain a simple 3-degenerate graph from $G'$ that has maximum degree 5. We replace the vertices of $G'$ each with a sufficiently large uncuttable $H_i$, as mentioned in Lemma \ref{lem:uncuttableH''}. 
The adjacencies of a vertex $v \in G'$ are distributed among the degree 4 and degree 2 vertices of $H_i$. 
Since the maximum degree of $G'$ is $4m$ (where $m$ is the number of clauses in $\mathcal F$), setting $i \geq 4m/7$ suffices.

The constructed graph is 3-degenerate. To observe this, 
note that by Lemma \ref{lem:uncuttableH''}, the graph $H_i$ contains 4 vertices of degree 2. So even with the additional adjacencies from $G'$, these vertices have degree 3. 
We can remove the degree 2 vertices in $H_i$ repeatedly and
arrive at a copy of $H$ disconnected from the rest of the graph. Since $H$ is 3-degenerate,
we can strip away the entire gadget. Once all the gadgets are removed, the resulting graph is empty. 
Thus, we get an uncuttable 3-degenerate graph $G$ of maximum degree 5 such that $G$ has a CF-cut if and only if the clean 3-SAT formula $\mathcal F$ is satisfiable.
\qed
\end{proof}

\section{An Uncuttable Class of Planar Graphs}
In this section, we present a class of planar graphs with a  1-regular conflict assignment that is uncuttable. Although we have already presented a class of uncuttable planar graphs with 1-regular conflicts in Lemma \ref{lem:planaruncuttable}, we present one more class of such graphs to answer the questions asked by \cite{rauch2025conflict} regarding the CF-cut on planar graphs. We first define prism graphs. 
\begin{definition}[Prism Graph]
    The \emph{prism graph} $P_{2t}$ on $2t$ vertices contains two $t$-cycles $v_1, v_2, \ldots, v_t$ and $v_{t+1}, v_{t+2}, \ldots, v_{2t}$, with additional edges $v_i v_{t+i}$ for all $1 \leq i \leq t$. 
\end{definition}

Prism graphs on 8 and 12 vertices are depicted in Fig. \ref{fig:Prism_Graphs}.

\begin{figure}[htbp]
    \centering
    \begin{subfigure}[b]{0.45\textwidth}
        \centering
        \includegraphics[width=\textwidth]{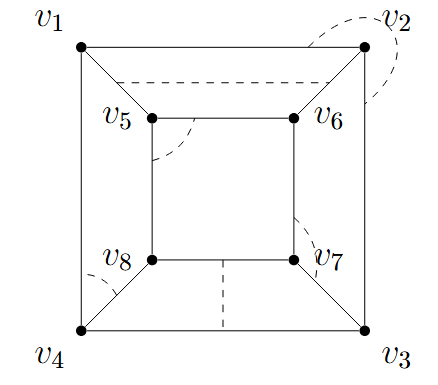}
        \caption{}
        \label{fig:Prism_on_Rectangle}
    \end{subfigure}
    \hfill
    \begin{subfigure}[b]{0.45\textwidth}
        \centering
        \includegraphics[width=\textwidth]{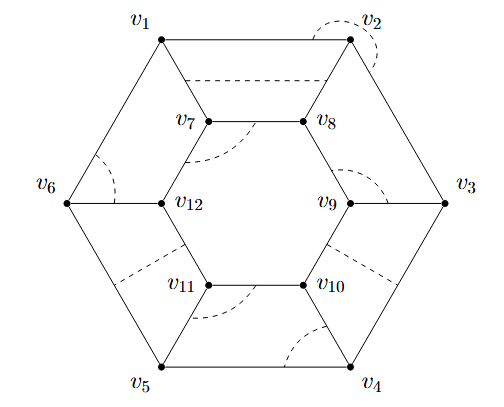}
        \caption{}
        \label{fig:Prism_on_Hexagon}
    \end{subfigure}
    \caption{(a) Prism Graph on $4t$ vertices for $t=2$ with 1-regular conflicts and (b) Prism Graph on $4t$ vertices for $t=3$ with 1-regular conflicts.}
    \label{fig:Prism_Graphs}
\end{figure}

The main result of this section is stated below. 

\begin{theorem}
\label{thm:easy_planar} 
Let $G = P_{4t}^*$ be the planar dual of the prism graph on $4t$ vertices, where $t \geq 2$. There exists a 1-regular conflict graph $\widehat{G}$ for which the graph $G$ does not have a conflict-free cut.
\end{theorem}

We first give a proof overview. Note that every minimal cut in a planar graph is a cycle in its planar dual graph. 

\begin{observation}\label{obs:planardual}
    Every minimal cut in a planar graph $G$ is a cycle in its planar dual $G^*$, and vice versa. 
\end{observation}

We say that a cycle in a graph is \emph{conflict-free} (CF-cycle, in short) if it does not contain conflicting edges. 
By the above observation, 
finding a CF-cut in a planar graph $G$ is equivalent to finding a conflict-free cycle in its planar dual $G^*$. 
We show that conflicts can be assigned for prism graphs $P_{4t}$, for $t \geq 2$, in such a way that there are no conflict-free cycles. 
We use two ``building blocks'': a graph $H$ on 10 vertices and a graph $H'$ on 6 vertices, both with associated conflicts. Depending on the value of $t$, the graphs $P_{4t}$ are constructed using multiple copies of $H$ and possibly a copy of $H'$.

\begin{figure}[htbp]
    \centering
    \begin{subfigure}[b]{0.55\textwidth}
        \centering
        \includegraphics[width=\textwidth]{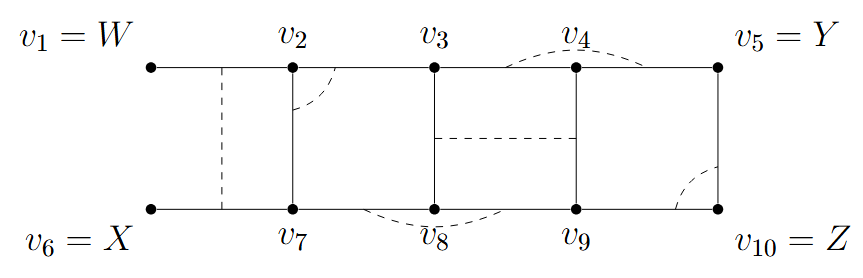}
        \caption{}
        \label{fig:Prism_Component-1}
    \end{subfigure}
    \hfill
    \begin{subfigure}[b]{0.4\textwidth}
        \centering
        \includegraphics[width=\textwidth]{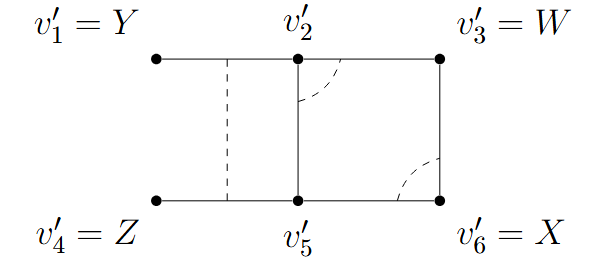}
        \caption{}
        \label{fig:Prism_Component-2}
    \end{subfigure}
    \caption{(a) Subgraph $H$ with 1-regular conflicts and (b) Subgraph $H'$ with 1-regular conflicts.}
    \label{fig:Prism_Components}
\end{figure}

\begin{lemma} 
\label{lem:cf-cycle}
Let $t\geq 2$. Using subgraphs $H$ and $H'$ with 1-regular conflicts, as shown in Figure \ref{fig:Prism_Components}, we can construct a prism graph $P_{4t} = (V, E)$ on $4t$ vertices with 1-regular conflicts without a CF-cycle.
\end{lemma}

\begin{proof} We refer to the vertices $v_1, v_5, v_6$, and $v_{10}$ of Figure \ref{fig:Prism_Component-1} as entry or exit vertices and remaining vertices as internal vertices. There are only 6 simple cycles in $H$ and it can be easily verified that none of them are CF-cycles. Moreover, there is no \emph{conflict-free path}, i.e. a path without two edges conflicting, starting from $v_6$ and ending at $v_1, v_5$, or $v_{10}$. 
Similarly, there is no conflict-free path that starts at $v_5$ and ends at $v_1$. There are exactly two conflict-free paths starting at $v_5$ and ending at $v_{10}$, namely $v_5 \text{--} v_4 \text{--} v_9 \text{--} v_{10}$ and $v_5 \text{--} v_{10}$. Similarly, there are two conflict-free paths from $v_1$ to $v_{10}$, namely $v_1 \text{--} v_2 \text{--} v_3 \text{--} v_4 \text{--} v_9 \text{--} v_{10}$ and $v_1 \text{--} v_2 \text{--} v_3 \text{--} v_8 \text{--} v_9 \text{--} v_{10}$. 

For even numbers $t$, we construct a $P_{4t}$ by attaching $t/2$ copies of $H$ in a sequential manner. The vertices $v_5$ and $v_{10}$ of a copy of $H$ are identified respectively with the vertices $v_1$ and $v_6$ of the next copy. Finally, the vertices $v_5$ and $v_{10}$ of the $(t/2)$-th copy of $H$ are identified with the vertices $v_1$ and $v_6$ of the first copy. So overall, we have $8 \cdot (t/2) = 4t$ vertices. 

We now argue that the constructed $P_{4t}$ has no conflict-free cycle. As noted above, there are no conflict-free cycles in a single copy of $H$. The only other possibility is that a conflict-free cycle spans multiple copies of $H$. This is not possible because there are no conflict-free paths from $v_1$ to $v_6$ in a copy of $H$. Hence, a conflict-free cycle cannot ``enter'' a copy of $H$ and ``loop back''. So the cycle has to ``enter'' through $v_1$ in a copy of $H$ and ``exit'' through $v_{10}$ of the same copy. However, this path starts at $v_6$ in the next copy of $H$ and hence cannot extend further.

Now, the only remaining case is when $t$ is odd. Let us label the vertices $v_1, v_5, v_6$ and $v_{10}$ in $H$ as $W, Y, X$ and $Z$ respectively. Similarly, we label the vertices $v_1', v_3', v_4'$ and $v_{6}'$ in $H'$ as $Y, W, Z$ and $X$ respectively. We now construct a $P_{4t}$ by attaching $\lfloor t/2 \rfloor$ copies of $H$ in a sequential manner as above, plus one copy of $H'$ such that the vertices marked $Y$ and $Z$ in the $\lfloor t/2 \rfloor$'th copy of $H$ are identified respectively with the vertices marked $Y$ and $Z$ in the copy of $H'$.
Similarly, the vertices marked $W$ and $X$ in the first copy of $H$ are identified, respectively, with the vertices marked $W$ and $X$ in the copy of $H'$.

It can be seen that a single copy of $H'$ has no conflict-free cycles (see in Figure \ref{fig:Prism_Component-2}).  
So any conflict-free cycle has to use at least one copy of $H$. We had noted earlier that any such cycle has to ``enter'' through $v_1$ in a copy of $H$ and ``exit'' through $v_{10}$ of the same copy. If the adjacent graph is a copy of $H$, then the same argument as in the case of even $t$ applies, and hence the cycle cannot extend further or loop back. 
If the adjacent graph is a copy of $H'$, then the path has to ``enter''  the copy of $H'$ through $v_4'$. It is easy to verify that there are no conflict-free paths in $H'$ that start at $v_4'$ and end at $v_1'$ or $v_3'$. Hence, the path has to ``exit'' the copy of $H'$ through $v_6'$. This path would have to ``enter'' a copy of $H$ through $v_6$, but such paths cannot exit $H$. 
This completes the proof for the case when $t$ is odd.
\qed
\end{proof}


\begin{proof}[Proof of Theorem \ref{thm:easy_planar}] 
We can use Lemma \ref{lem:cf-cycle} to construct a prism graph $P_{4t}$ on $4t$ vertices without conflict-free cycles. The planar dual of $P_{4t}$, will therefore have no conflict-free cuts. 

Since $P_{4t}$ is 3-regular, the planar dual $G$ will be a planar graph whose regions are all triangles. In other words, $G$ is a triangulation graph with 1-regular conflicts. Since a planar triangulation graph is also known as a maximal planar graph, we have a class of maximal planar graphs with 1-regular conflicts that is uncuttable. \qed
\end{proof}

\begin{corollary}\label{cor:octahedron} The planar dual of the prism graph on $8$ vertices, constructed using Lemma \ref{lem:cf-cycle}, is a 4-regular uncuttable planar graph with 1-regular conflicts. The planar dual in this case is an octahedron graph. 
\end{corollary}

\begin{proof} This claim is easily verifiable when we take the planar dual of the prism graph of the Figure \ref{fig:Prism_on_Rectangle}, which has no CF-cycle with 1-regular conflicts.\qed
\end{proof}
\paragraph{} Corollary \ref{cor:octahedron} also answers a question of \cite{rauch2025conflict} in which they asked about finding a 4-regular uncuttable graph with 1-regular conflicts that is not the square of an odd cycle.

\section{Future Work}
We conclude with some open directions: 
\begin{enumerate}
    \item The complexity of the CF-cut problem when $G$ is 4-regular (not necessarily planar) with 1-regular $\widehat{G}$.
    \item The complexity of the  CF-cut problem for maximal planar graphs $G$ with 1-regular $\widehat{G}$.
    \item The complexity of counting CF-cuts for a given graph $G$ with 1-regular $\widehat{G}$. It would also be an interesting direction for specific classes of $G$, such as when $G$ is 4-regular and planar. There are existing results on the enumeration of matching cuts \cite{golovach2021}.
    \item Taking a cue from the parameterized complexity results on matching cuts \cite{feghali2025journal,komusiewicz2020}, it would be interesting to study the complexity of CF-cuts with respect to different structural parameters of $G$. 
\end{enumerate}

\bibliographystyle{plain}
\bibliography{references}

\end{document}